\documentclass[a4paper,10pt]{article}
\usepackage[utf8]{inputenc}
\usepackage{amsthm}
\usepackage{amsmath}
\usepackage{authblk}
\usepackage{mathtools}
\usepackage[mathscr]{euscript}
\usepackage{cite}
\usepackage{amssymb}
\usepackage{nicefrac}
\usepackage{graphicx} 
\usepackage{float}
\usepackage[ruled,vlined]{algorithm2e}
\include{pythonlisting}
\newtheorem{theorem}{Theorem}

\newtheorem{proposition}{Proposition}

\newcommand\inv[1]{#1\raisebox{1.15ex}{$\scriptscriptstyle-\!1$}}

\newtheorem{remark}{Remark}
\title{DSPG: Decentralized Simultaneous Perturbations Gradient Descent Scheme
}
\author{Arunselvan Ramaswamy
\thanks{This work was supported by the German Research Foundation (DFG)
(project number 315248657)}
\thanks{Department of Electrical Engineering and Computer Science, Paderborn University, Paderborn - 33102, Germany
        {\tt\small arunr@mail.upb.de}}
}
\begin{document}
\maketitle
\begin{abstract}
Distributed descent-based methods are an essential toolset to solving optimization problems in multi-agent system scenarios. Here the agents seek to optimize a global objective function through mutual cooperation. Oftentimes, cooperation is achieved over a wireless communication network that is prone to delays and errors. There are many scenarios wherein the objective function is either non-differentiable or merely observable. In this paper, we present a cross-entropy based distributed stochastic approximation algorithm (SA) that finds a minimum of the objective, using only samples. We call this algorithm \textit{Decentralized Simultaneous Perturbation Stochastic Gradient, with Constant Sensitivity Parameters} (DSPG). This algorithm is a two fold improvement over the classic Simultaneous Perturbation Stochastic Approximations (SPSA) algorithm. Specifically, DSPG allows for (i) the use of old information from other agents and (ii) easy implementation through the use simple hyper-parameter choices. We analyze the biases and variances that arise due to these two allowances. We show that the biases due to communication delays can be countered by a careful choice of algorithm hyper-parameters. The variance of the gradient estimator and its effect on the rate of convergence is studied. We present numerical results supporting our theory. Finally, we discuss an application to the stochastic consensus problem.

\end{abstract}
\section{INTRODUCTION}\label{sec_intro}

Multi-agent systems (MAS) such as traffic networks, smart grids and robotic swarms are distributed systems with multiple interacting agents.  These agents compete or cooperate with each other to achieve a global objective. This objective is usually achieved by solving interrelated local optimization problems using both local and global/outside information. A wireless network, with possibly time varying reliabilities and topologies, is typically used for information exchange. In this paper, we present a cross-entropy based gradient method, DSPG, that is robust to changing network topologies, unbounded stochastic delays and errors. This algorithm does not require gradient information, instead it uses random samples of the objective, and possibly old information from other agents to calculate descent directions at every stage. Further, it does not require any synchronization between the agent-clocks.

Our algorithm, DSPG, has wide-ranging applications including multi-agent learning (MAL), stochastic consensus and decentralized stochastic optimization. In the past, distributed optimization was tackled within the setting of optimal control to solve the stochastic consensus problem. Recall that, the goal in a consensus problem is to find a common "control action" that minimizes some "network cost". Stochastic consensus with unbounded gradient delays was tackled in \cite{sirb2018decentralized}, and a subgradient method for the consensus problem was presented in \cite{nedic2009distributed}. The reader is referred to \cite{nedic2018distributed} for a thorough literature survey of distributed optimization algorithms for control problems. In this paper, we shall illustrate that DSPG can be used to solve the stochastic consensus problem with unbounded stochastic delays and unknown gradients. In other words, we endeavour to combine the results of \cite{sirb2018decentralized} and \cite{nedic2009distributed}. In what follows, we introduce Simultaneous Perturbations Stochastic Approximations (SPSA), \cite{spall1992multivariate}, a popular cross-entropy based descent algorithm which forms a basis for DSPG.



\subsection{SPSA and Approximate Gradient Methods}
Generally speaking, the goal in optimization is to find $x^*$ that minimizes a given objective function $f: \mathbb{R}^d \to \mathbb{R}$ i.e., find $x^* = \min \limits_{x \in \mathbb{R}^d}\ f(x)$. In today's data-driven world, one often encounters optimization problems wherein the objection functions are either non-differentiable or their gradients are incomputable. For example. let us consider machine learning algorithms that use convolutional neural networks (CNNs) to parameterize given objective functions. Due to the presence of a max-pooling layer, analytic backpropagation is not possible. Here, approximate gradient methods such as SPSA play an important role, since they only require a few forward passes of the CNN. The reader is referred to \cite{kiefer1952stochastic}, \cite{spall1992multivariate} and \cite{ramaswamy2018analysis} for different approximate gradient algorithms.


We focus on SPSA since our algorithm is related to it. Broadly speaking, SPSA estimates the minimum of an objective iteratively through a sequence of stochastic updates.
Specifically, the i$^{th}$ component of the mimimum-estimate, at time $n$, is updated as:

\begin{equation} \label{intro_spsa_eq}
x_{n+1}(i) = x_n(i) - \gamma_n \left[ \frac{f(x_n + c_n \Delta_n) - f(x_n - c_n \Delta_n)}{2 c_n \Delta_n(i)} \right], \text{ where}
\end{equation}
(i) $x_n \in \mathbb{R}^d$ ($d \ge 1$) is the $n^{th}$ estimate of the minimum. \\
(ii) $\Delta_n = (\Delta_n(1), \ldots, \Delta_n(d))$ is a vector with independent symmetric Bernoulli random components, i.e., $\Delta_n (i)  = \pm 1$ w.p. $\nicefrac{1}{2}$, $1 \le i \le d$. \\
(iii) The sensitivity parameters $\{c_n\}_{n \ge 0}$ and the step-sizes $\{\gamma_n\}_{n \ge 0}$ (learning rates) are such that $c_n \downarrow 0$ and $\sum \limits_{n \ge 0} \frac{\gamma_n^2}{c_n ^2} < \infty$.

Condition (iii) is key to showing that the minimum-estimate-sequence generated by \eqref{intro_spsa_eq} converges to a minimum of $f$. However, (iii) affects the choice of the learning rates, thereby reducing the applicability of SPSA to practical problems. An important extension to SPSA that boosts its applicability is presented in \cite{ramaswamy2018analysis}. This extension is called SPSA-C and is obtained by allowing for constant sensitivity parameters, .i.e., $c_n = c$ for all $n \ge 0$. The reader is referred to \cite{prashanth2017adaptive} for yet another extension that reduces the number of samples needed.
\subsection{Our Contributions}
{\it In this paper, we present DSPG, an approximate gradient algorithm for optimization problems that arise in asynchronous multi-agent settings}. It is obtained by extending SPSA-C \cite{ramaswamy2018analysis} to allow for distributed computations. As in \cite{ramaswamy2018analysis}, we let $c_n = c$ for all $n \ge 0$.
The i$^{th}$ agent makes an update similar to \eqref{intro_spsa_eq}.  For this update, it obtains $x_n(j)$, $j \neq i$, over a wireless network. As stated earlier, this network is prone to delays and errors. Hence agent-i may perform the n$^{th}$ update using $x_{m(j)}(j)$, where $m(j) < n$ and $j \neq i$. {\it In other words, agent-i performs updates using the last known information from other agents}.


Recall that the agents are all fully asynchronous. The relative frequency of the various agent updates affects the limiting point. We control this via the use of Borkar's balanced step-size sequence \cite{borkar1998asynchronous}. We show that DSPG converges to a near optimal point, whose optimality is affected by the size of the sensitivity parameter ($c$) and the communication errors. We also briefly discuss a constant step-size version of DSPG. We discuss the relationship between step-sizes, sensitivity parameters and the communication frequency. To summarize, we show that (a) convergence of DSPG is unaffected by stochastic delays that are moment bounded, (b) the size of the constant sensitivity parameter affects the neighborhood of convergence and (c) the errors due to delays are asymptotically in the order of the step-sizes. In the following section, we shall formally setup the optimization problem at hand.
\section{THE PROBLEM SETUP}\label{sec_setup}
In this paper, we consider a $d$-agent system with autonomously operating agents that are capable of exchanging information via an unreliable wireless network. All the agents endeavour to minimize local objective functions that require both local and global information. In particular, agent-i tries to find $x(i)^*$ that minimizes $F_i : \mathbb{R}^n \to \mathbb{R}$, where $x(i)^* \in \mathbb{R}^{d_i}$ and $\sum \limits_{i=i}^d d_i = n$. Note that agent-i requires estimates, $x(j)$ for $j \neq i$, from other agents in order to calculate $F_i$. Also, note that  $F_i$ is only accessible to agent-i. The global objective is therefore summarized as:
\begin{equation} \label{opt_prob}
\text{Find }x^* = (x^*(1), \ldots, x^*(d)) \text{ such that } x^* = \underset{x \in \mathbb{R}^d}{\text{ argmin }} F_i(x) \ \forall i.
\end{equation}

For clarity in presentation, without loss of generality, we assume that $d_i = 1$ for all $1 \le i \le d$. It may be the case that some or all of the $F_i$s are equal. If we consider the special case when $F_i \equiv F$ for all $1 \le i \le d$, then the problem is to find a minimizer of $F$ in a distributed decentralized manner. As stated previously, the $F_i$s are merely observable and their gradients unavailable. {\it Moving forward, we assume that there is at least one simultaneous minimizer for all $F_i$s.} In the following section we present DSPG, our cross-entropy based distributed approximate gradient algorithm to solve \eqref{opt_prob}.


%
\section{DSPG: THE ALGORITHM}

As DSPG is based on SPSA-C \cite{ramaswamy2018analysis}, a centralized approximate gradient method, we start with a quick recollection of it. The i$^{th}$ component of the minimum-estimate is updated as:
\begin{equation}
\label{dsgp_spsac_eq}
x_{n+1}(i) = x_n(i) - \gamma_n \left[ \frac{F(x_n + c \Delta_n) - F(x_n - c \Delta_n)}{2 c \Delta_n(i)} \right].
\end{equation}
 The sensitivity parameter ‘c' is used to control the bias and variance of the estimation errors. Getting back to the problem at hand, to find $x^*$ that simultaneously minimizes $F_i$ for all $1 \le i \le d$ in a distributed manner, we propose the following update to the agent-i estimate at time $n$:
 \begin{equation}
\label{dspg_dspg_eq}
x_{n+1}(i) = x_n(i) - \gamma_n  \frac{F_i \left( \begin{matrix} x_{n - \tau_{1i}(n)}(1) + c \Delta_n^i(1) \\ \vdots \\ x_{n - \tau_{di}(n)}(d) + c \Delta_n^i(d) \end{matrix} \right) - F_i \left( \begin{matrix} x_{n - \tau_{1i}(n)}(1) - c \Delta_n^i(1) \\ \vdots \\ x_{n - \tau_{di}(n)}(d) - c \Delta_n^i(d) \end{matrix} \right)}{2 c \Delta_n^i (i)},
\end{equation}
where (i) $\Delta_n ^i$ is a random vector generated by agent-i, at time $n$, such that its components are independent symmetric Bernoulli random variables; (ii) $0 \le \tau_{ji}(n) \le n$ is a delay random variable denoting the age of the latest agent-j estimate that is available at agent-i. In other words, agent-i does not have access to $x_m(j)$ for $n - \tau_{ji}(n) < m \le n$. \textit{The reader may note that while \eqref{dspg_dspg_eq} suggest that all agents use a single sensitivity parameter $c$, this is not true. We use a common $c$ merely to reduce clutter.} To summarize, agent-i updates its local estimate of $x^*(i)$ using the gradient estimator given by:
\begin{equation} \label{dspg_est_eq}
\frac{F_i \left( \begin{matrix} x_{n - \tau_{1i}(n)}(1) + c \Delta_n^i(1) \\ \vdots \\ x_{n - \tau_{di}(n)}(d) + c \Delta_n^i(d) \end{matrix} \right) - F_i \left( \begin{matrix} x_{n - \tau_{1i}(n)}(1) - c \Delta_n^i(1) \\ \vdots \\ x_{n - \tau_{di}(n)}(d) - c \Delta_n^i(d) \end{matrix} \right)}{2 c \Delta_n^i (i)}, 
\end{equation}
which in turn utilizes the latest estimates, $x^*(j)$ for $j \neq i$, from other agents. The DSPG algorithm given by Algorithm $1$ is a natural consequence of the above discussion. In Section~\ref{sec_asmp}, we discuss the assumptions under which Algorithm 1 converges to a common minimum of $F_i$, $1 \le i \le d$.

\begin{algorithm}
\SetAlgoLined
{\bf Initialization:} Sensitivity parameters are all set to $c$.\\
 \For{every timestep (of the local-clock)}{
 \For{every agent-j ($\neq$ i)}{
  \eIf{new estimate $x_{new}(j)$ is received from agent-j}{$x_n(j) \leftarrow x_{new}(j)$}
  {$x_n(j) \leftarrow x_{n-1}(j)$}
 }
 Generate $\Delta_n = (\Delta_n^i (1), \ldots, \Delta_n^i (d))$ such that $\Delta_n^i (j)$s are independent symmetric Bernoulli.\\
 Perform a gradient descent step in the $i^{th}$ direction as follows:\\
 $x _{n+1}(i) \leftarrow x _n(i) - \gamma_n \left( \frac{F_i(x_n + c \Delta_n^i) - F_i(x_n - c \Delta_n^i )}{2 c \Delta_n^i (i)} \right)$
 }
 \caption{DSPG (as executed by agent-i)}
\end{algorithm}

Clearly \eqref{dspg_est_eq} is an approximation for $\nicefrac{\partial F_i(x_n)}{\partial x(i)}$. In other words, ideally, if all gradient computations were possible, and all communications instantaneous and perfect, then $\nicefrac{\partial F_i(x_n)}{\partial x(i)}$ would be used instead of \eqref{dspg_est_eq} in \eqref{dspg_dspg_eq}. Spall \cite{spall1992multivariate} showed that the biases associated with gradient errors asymptotically vanish, provided the sensitivity parameters also vanish. Since the sensitivity parameters are fixed in our case, we expect that our gradient errors are asymptotically biased. In the following section, we bound this bias and the associated variance. This is important, since bias affects how close Algorithm 1 gets to a minimum. Also, the variance affects the rate of convergence and smoothness.
\subsection{Bounding the Gradient-Error Bias and Variance}\label{sec_bias}
Let us expand  $F_i(x_n + c \Delta_n^i)$ and $F_i(x_n - c \Delta_n^i)$ around $x_n$, using Taylor's theorem for multiple variables \cite{folland2005higher}:
\begin{equation}
\label{taylor_series}
F_i(x_n \pm c \Delta_n^i) = F_i(x_n) \pm \sum \limits_{j=1}^d \frac{\partial F_i(x_n)}{\partial x(j)} c \Delta_n^i(j) + o\left(\lVert c \Delta_n^i \rVert_2 ^2 \right).
\end{equation}
Using the above Taylor series expansion, we replace the terms in the numerator of the gradient approximation, \eqref{dspg_est_eq}, in the last step of Algorithm $1$  to get:
\begin{equation}
\begin{split}
\label{grad_taylor1}
 &\frac{F_i(x_n + c \Delta_n^i) - F_i(x_n - c \Delta_n^i)}{2c \Delta_n^i(i)} =\\  &\inv{\Delta_n^i(i)} \left(\sum \limits_{j=1, j \neq i}^d \frac{\partial F_i(x_n)}{\partial x(j)} \Delta_n^i (j) \right) + \frac{\partial F_i(x_n)}{\partial x(i)} +  o\left(\lVert c \Delta_n^i \rVert_2 ^2 \right).
 \end{split}
\end{equation} 
Since $\Delta_n^i(i)$ is symmetric Bernoulli, $\inv{\Delta_n^i(i)} = \Delta_n^i(i)$ almost surely. Hence, $\inv{\Delta_n^i(i)}$ is also symmetric Bernoulli and independent of $\Delta_n^i(j)$ for $j \neq i$.
Further, $\mathbb{E} \Delta_n^i(j) = 0$ and $\mathbb{E} \inv{\Delta_n^i(i)} = 0$ for all $i \neq j$. 
Hence, 
\[
\mathbb{E}\left[\inv{\Delta_n(i)}\right] \left(\sum \limits_{j=1, j \neq i}^d \frac{\partial F_i(x_n)}{\partial x(j)} \mathbb{E}\left[\Delta_n(j)\right] \right) = 0.
\]
Using this observation and taking expectations on both sides \eqref{grad_taylor1}, we get:
\begin{equation}
\label{grad_taylor2}
\mathbb{E}\left[ \frac{F_i(x_n + c \Delta_n^i) - F_i(x_n - c \Delta_n^i)}{2c \Delta_n^i(i)} \right] =  \frac{\partial F_i(x_n)}{\partial x(i)} +  o\left(c^2 d\right),
\end{equation} 
In other words, the biases are asymptotically smaller than $c^2 d$. This gives us a rough idea in choosing the sensitivity parameter. Clearly one also needs to factor in the dimension $d$ of the multi-agent system, when choosing $c$. 

Now, we are ready to bound the variance of the gradient estimator. Hence, we consider:
\begin{equation} \label{grad_var}
\mathbb{E} \left[\frac{F_i(x_n + c \Delta_n^i) - F_i(x_n - c \Delta_n^i)}{2c \Delta_n^i(i)} -  \frac{\partial F_i(x_n)}{\partial x(i)} \right]^2.
\end{equation}
Using arguments that are similar to the ones used to bound the bias, we get that \eqref{grad_var} equals the following:
\[
var\left[\inv{\Delta_n^i(i)}\right] \left(\sum \limits_{j=1, j \neq i}^d \frac{\partial F_i(x_n)}{\partial x(j)} ^2 var\left[\Delta_n^i(j)\right] \right)  + o(c^4 d^2).
\]
Since $var\left[\inv{\Delta_n^i(i)}\right] = var\left[\Delta_n^i(j)\right] = 2$, we get the following bound for the variance:
\begin{equation} \label{grad_est_var}
4\sum \limits_{j=1, j \neq i}^d \frac{\partial F_i(x_n)}{\partial x(j)} ^2 + o(c^4 d^2).
\end{equation}
Although the variance does not affect the limiting point, it does affect the rate of convergence. Further, one can expect aberrant transient behavior. \textit{This aberrance has been captured in the experiments conducted. The reader is referred to Figure \ref{fig2} in Section \ref{sec_numerical} for details.}

In this section, we presented the DSPG algorithm. We bounded the asymptotic bias and variance of the gradient errors as a function of the sensitivity parameter $c$. In the following section, we present a convergence analysis of Algorithm 1, and understand the interplay between the limiting point, asymptotic error bias and $c$. 

\section{ANALYZING DSPG}\label{sec_convergence}
The tools and techniques associated with stochastic approximation algorithms (SAs) are key to the analysis of Algorithm 1. Broadly speaking, the field of SAs includes tools and techniques required to develope and analyze stochastic algorithms that have an iterative structure. The first SA was developed by Robbins and Monro \cite{robbins1951stochastic} to solve the root finding problem. In the mid 90s, important contributions were made by \cite{benaim1996dynamical}, \cite{benaim1996asymptotic} and \cite{borkar2000ode}. Recently, \cite{benaim2005stochastic} \cite{ramaswamy2016generalization} and others have extended the theory to allow set-valued objective functions. These extensions find applications in the analysis of deep reinforcement learning algorithms \cite{ramaswamy2018stability}.

It may be noted that the analysis presented herein is similar to the one presented in Section~4 of \cite{ramaswamy2018asynchronous}. To avoid redundency, we only present the new arguments, in addition to an overview of the proof. Before we proceed, let us recast Algorithm 1 and \eqref{dspg_dspg_eq} as the following SA:
\begin{multline}
\label{dspg_sa}
x_{n+1}(i) = x_n (i) - \gamma_{\nu(n, i)} I(i \in Y_n)  \hat{g} (x_{n - \tau_{1i}}(1), \ldots, x_{n - \tau_{di}}(d))(i), \Delta_n^i), \text{ where}
\end{multline}
(i) 
 $
\hat{g}(x)(i) = \frac{F_i(x + c \Delta_n^i) - F_i(x - c \Delta_n^i)}{2c \Delta_n^i(i)}
$
is the i$^{th}$ component of gradient estimator \eqref{dspg_est_eq}. \\
(ii) $\nu(n,i)$ is the number of times that agent-i is updated upto time $n$.\\
(iii) $Y_n \subset \{1, \ldots, d\}$ is the subset of agents active at time $n$.

All other terms are as before. Note that $\{Y_n\}_{n \ge 0}$ and $\{\nu(n,i)\}_{n \ge 0, \ 1 \le i \le d} $ are used to account for the complete lack of synchronization between the agents involved. It must be noted that we make certain causal assumptions on the relative update frequency of the various agents, see (A5) below. However, this assumption does not require any synchronization during implementation.
\subsection{Assumptions} \label{sec_asmp}
Before we list the assumptions that ensure convergence of Algorithm 1, we make a quick but important note on notation.\\
\vspace*{.005cm}\\
{\bf Quick note on notation:} {\it The following assumptions and analysis involves associating an o.d.e. with \eqref{dspg_sa}. To avoid confusion, note that we will use $x^d$ to represent the $d^{th}$ component of vector $x$, instead of the previously used $x(d)$. In this section, if we use $x(t)$, then the $t$ denotes time. For example, we use $\dot{x}(t) = f(x(t))$, $t \ge 0$, to denote an o.d.e., and $(x^1, \ldots, x^d)$ for vector $x$.}
\\
\vspace*{.005cm}\\
Below, we list the assumptions involved.
\begin{itemize}
\item[{\bf (A1)}] $F_i$ is Lipschitz continuous for all $1 \le i \le d$. Without loss of generality, $L$ is the Lipschitz constant associated with all of them.
\item[ \textbf{(A2)}] (i) $0 \le \gamma_n \le 1$ $\forall \ n$, (ii) $\gamma_n \le \kappa \ \gamma_m$ for $m \le n$ and fixed $\kappa > 0$, (iii) $\gamma_n \in o(n^{-\eta})$ where $\eta \in (1/2, 1]$, (iv) $\sum \limits_{n \ge 0} \gamma_n = \infty$, (v) $\sum \limits_{n \ge 0} \gamma_n ^2 < \infty$ and (vi) $\limsup \limits_{n \to \infty} \sup \limits_{y \in [x, 1]} {\Huge \frac{\gamma_{[yn]}}{\gamma_n} }< \infty$ for $0 < x \le 1$.
\item[{\bf (A3)}] $\sup \limits_{n \ge 0} \lVert x_n \rVert < \infty$ a.s.
\item[{\bf (A4)}] $x^*$ is the unique global asymptotic stable equilibrium point of $\dot{x}(t) = \begin{bmatrix}
\nicefrac{\partial F_1 (x(t))}{\partial x^1} \\
\vdots \\
\nicefrac{\partial F_d (x(t))}{\partial x^d}
\end{bmatrix},
$
where $x^* = \underset{x \in \mathbb{R}^d}{\text{argmin}} \left\{ \max \limits_{1 \le i \le d} F_i(x) \right\}$.
\item[{\bf (A5)}] (i) $\liminf \limits_{n \to \infty} \frac{\nu(n,i)}{n} > 0$ for all $1 \le i \le d$ and (ii) there is a random variable $\overline{\tau}$ that stochastically dominates $ \tau_{ij}(n)$ for all $1 \le i,j \le d$ and $n \ge 0$, such that $\mathbb{E}  \tau_{ij}(n) ^2 < \infty$.
\end{itemize}
 
Assumption (A4) essentially ensures that the global optimization problem \eqref{opt_prob} has a solution. It may be noted that the uniqueness assumption can be readily relaxed. Then, one cannot ensure convergence to a particular minimizer. The first part of assumption (A5) states that the number of updates per agent are in the ``order of $n$''. The second part states that the delay random variables while having unbounded support are moment bounded. The Lipschitz continuity of the gradient estimator $\hat{g}$ is a direct consequence of (A1). This claim is a direct consequence of the following two inequalities:
\[
\left| \hat{g}(x)(i) - \hat{g}(y)(i) \right| \le \frac{L}{c} \lVert x - y \rVert, \text{ and}
\]
\[
\lVert \hat{g}(x) - \hat{g}(y) \rVert \le \frac{\sqrt{d} L}{c} \lVert x - y \rVert.
\]

Before stating the main convergence theorem, we present the following technical lemma. It states that the errors due to communication delays are asymptotically in the order of the step-size. Since we use diminishing step-sizes, they vanish over time. {\it An important implication of this lemma is the following: suppose one were to use constant step-size/ learning rate in DSPG, then these errors do not vanish over time.}

\begin{proposition}
If DSPG is bounded almost surely, then the errors due to bounded communication delays vanish asymptotically.
\end{proposition}
\begin{proof}
We begin the proof with a quick recollection of our note on notation. Given a vector $x$, its i$^{th}$ component is denoted using $x^i$, instead of $x(d)$. We are required to show that $\left| \hat{g}(x_n)^i - \hat{g} (x_{n - \tau_{1i}(n)}^1, \ldots, x_{n - \tau_{di}(n)}^d)^i) \right| \to 0$ as $n \to \infty$. It follows from the Lipschitz continuity of $\hat{g}$ that:
\begin{multline}
\left| \hat{g}(x_n)^i - \hat{g} (x_{n - \tau_{1i}(n)}^1, \ldots, x_{n - \tau_{di}(n)}^d)^i) \right| \le \frac{L}{c} \lVert x_n - (x_{n - \tau_{1i}(n)}^1, \ldots, x_{n - \tau_{di}(n)} ^d) \rVert.
\end{multline}
Now, we focus on bounding $\left| x_n^j - x_{n - \tau_{ji} ^d}^j \right|$ for $1 \le j \le d$. It follows from \eqref{dspg_sa}, the Lipschitz continuity of $\hat{g}$, and the almost sure boundedness of the iterates ($\sup \limits_{n \ge 0} \lVert x_n \rVert < \infty$ a.s.) that
\[
\sum \limits_{m = n - \tau_{ji}(n)} ^{n-1} |x_{k+1}^j - x_k^j | \le
\frac{L}{2c} \sum \limits_{m = n - \tau_{ji}(n)} ^{n-1} \gamma_k C,
\]
for some fixed $0 \le C < \infty$. It follows from (A5) that $P(\tau_{ij}(n) > n^{\eta} i.o.) = 0$ for $1 \le i,j \le d$. Since $\gamma_n \in o(n^{- \eta})$, the RHS of the above inequality is in $o(1)$. Hence, $\left| \hat{g}(x_n)^i - \hat{g} (x_{n - \tau_{1i}(n)}^i, \ldots, x_{n - \tau_{di}(n)}^d)^i \right| \to 0$ as $n \to \infty$, as required.

\end{proof}
\subsection{Convergence Theorem} \label{sec_conv_thm}
We are now ready to state the main theorem of this paper. As stated earlier, we only provide an overview of the proof. The reader is referred to Section~4 of \cite{ramaswamy2018asynchronous} for details.
\begin{theorem}\label{main_theorem}
Under assumptions $(A1)-(A5)$, DSPG converges to a small  neighborhood of $x^*$. Further, this neighborhood depends on $c$, i.e., the neighborhood size grows as $c \uparrow \infty$.
\end{theorem}
\begin{proof}
First, let us rewrite DSPG/\eqref{dspg_sa} as:
\begin{equation} \label{dspg_eq}
x_{n+1} = x_n - \overline{\gamma}_n D_n \hat{g}(x_n), \text{ where}
\end{equation}
(i) $\overline{\gamma}_n = \max \limits_{1 \le i \le d} \gamma{\nu(n, i)} I (i \in Y_n)$ and\\ (ii) $D_n = \begin{bmatrix}
\frac{\gamma_{\nu(n, 1)} I (1 \in Y_n)}{\overline{\gamma}_n} & \dots & 0 \\
\vdots & \ddots & \vdots \\
0 & \dots & \frac{\gamma_{\nu(n, d)} I (d \in Y_n)}{\overline{\gamma}_n}
\end{bmatrix}.
$

Using arguments similar to those in Section $4$ of \cite{ramaswamy2018asynchronous}, we can show that \eqref{dspg_eq} tracks a solution to the ordinary differential equation (o.d.e.) given by $\dot{x}(t) = \Lambda(t) \hat{g}(t)$, where $\Lambda(t)$ is a diagonal matrix with entries in $[0,1]$, $t \ge 0$. Hence, the asymptotic behaviors of \eqref{dspg_eq} and $\dot{x}(t) = \Lambda(t) \hat{g}(t)$ are identical.

Since all the agents utilize the same step-size sequence, it follows from Theorem $3.2$ of \cite{borkar1998asynchronous} that \eqref{dspg_eq} (Algorithm $1$, DSPG) tracks a solution to 
\begin{equation} \label{dspg_b1}
\dot{x}(t) = 
\begin{bmatrix}
1/d & \dots & 0 \\
\vdots & \ddots & \vdots \\
0 & \dots & 1/d 
\end{bmatrix} \hat{g}(t).\end{equation}
In other words, $\Lambda(t) = \begin{bmatrix}
1/d & \dots & 0 \\
\vdots & \ddots & \vdots \\
0 & \dots & 1/d 
\end{bmatrix}$ for all $t \ge 0$. Further, the asymptotic behavior of \eqref{dspg_b1} and $\dot{x}(t) = \hat{g}(t)$ are identical since their driving vector fields are $1/d$ factor apart in each co-ordinate, see \cite{aubin2012differential}. Finally, any solution to $\dot{x}(t) = \hat{g}(t)$ can be viewed as a perturbation of some solution to 
\begin{equation} \label{dspg_track}
\dot{x}(t) = \begin{bmatrix}
\nicefrac{\partial F_1 (x(t))}{\partial x^1} \\
\vdots \\
\nicefrac{\partial F_d (x(t))}{\partial x^d}
\end{bmatrix}. \end{equation} This is because $\hat{g}(x_n)$ is an approximation for $\begin{bmatrix}
\nicefrac{\partial F_1 (x_n)}{\partial x^1} \\
\vdots \\
\nicefrac{\partial F_d (x_n)}{\partial x^d}
\end{bmatrix}$ with an approximation error of $o(c^2 d)$. 

Recall from (A4) that $x^*$ is the global asymptotic stable equilibrium of \eqref{dspg_track}. It follows from the upper semicontinuity of attractor sets that $\dot{x}(t) = \hat{g}(t)$ also has a globally asymptotic stable equilibrium set in a neighborhood of $x^*$, see \cite{aubin2012differential}. Further, the diameter of this neighborhood is directly proportional to $c$ (sensitivity parameter).

To summarize, DSPG converges to a small neighborhood of $x^*$, provided $c$ is small.\end{proof}
\section{NUMERICAL RESULTS} \label{sec_numerical}
Algorithm 1 (DSPG) can be used to estimate $x^* = (x(1)^*, \ldots, x(d)^*)$ in an interative manner such that $x^*$ minimizes $F_i$ for all $1 \le i \le d$. Since only agent-i has access to $F_i$, it updates its current estimate of $x(i)^*$ using estimates from other agents. Specifically, it performs the following update (last step of Algorithm 1) at time $n$: 
\begin{equation} \label{nr_eq}
x_{n+1}(i) = x_n(i) - \gamma_n \hat{g}(x_n)(i),
\end{equation}
where $\hat{g}(x_n)(i)$ is given by \eqref{dspg_est_eq}.
\subsection{Experiment Parameters}
We let $F_i := x^{T}A_i x$ such that $A_i$ is a randomly chosen $d \times d$ positive definite matrix, where $1 \le i \le d$. Then, $\hat{g}(x_n)(i)$ in \eqref{nr_eq} is given by:
\begin{multline*}
\hat{g}(x_n)(i) := \frac{(x_n + c \Delta_n^i)^{T}A_i (x_n + c \Delta_n^i) - (x_n - c \Delta_n^i)A_i(x_n - c \Delta_n^i)}{c \Delta_n^i(i)}. 
\end{multline*}
Clearly, the origin or the 0-vector is the required minimizer. We present experimental results for sensitivity parameters varied between $0.1$ and $1$, in steps of $0.1$. Note that , all agents use the same sensitivity parameter $c$. We repeat the experiments for $4$ and $10$ agent systems, and for various communication configurations.

Agent-i obtains estimates of $x(j)^*$, $j \neq i$, from agent-j over a wireless channel, modelled as an erasure channel with dropouts according to an i.i.d. Bernoulli random process. Every pair of agents is connected by two unidirectional erasure channels. This can be readily simulated by associating communication vectors $\psi_i(n) = (\psi_{1i}(n), \ldots, \psi_{di}(n))$, with agent-i for all $1 \le i \le d$, at time $n$. If agent-i successfully receives $x_n(j)$ from agent-j, then $\psi_{ji}(n) = 1$, else it is $0$. Further, $P(\psi_{ji}(n) = 1) = p_c$ and $P(\psi_{ji}(n) = 0) = 1- p_c$, for some fixed $0 < p_c \le 1$. We use the same success probability $p_c$ for all channels and present empirical results for $p_c$ between $0.2$ and $0.9$.


Algorithm $1$ is run for $20000$ iterations. For the first 5000 iterations, a constant step-size of $0.001$ is used. Thereafter we use diminishing step-sizes, starting with \nicefrac{1}{100} and vanishing at the rate of $\nicefrac{1}{n}$. {\it It must be noted that optimal step-size sequence that hastens convergence is highly problem dependent}.
\subsection{Exp. 1: $4$-agent system}
 \begin{figure}[H]
      \centering
      \includegraphics[width=.5\textwidth]{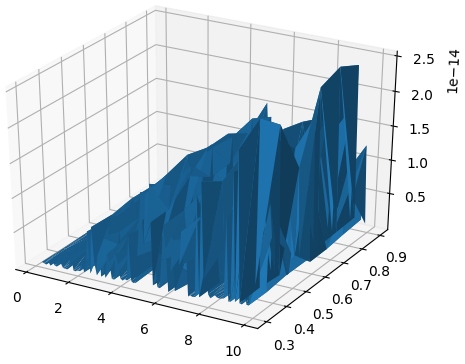}
      \caption{4-agent summary:  $0.1 \le c \le 10$ is plotted along $x$-axis, $0.3 \le p_c \le 0.9$ is plotted along $y$-axis and $\lVert (x_{20000}(1), \ldots, x_{20000}(4)) \rVert$ is plotted along  $z$-axis. Each point in the plot is obtained as an average of $20$ independent experiments.}
      \label{fig1}
   \end{figure}
   Figure \ref{fig1} illustrates the performance of DSPG in a $4$-agent setting for different values of the sensitivity parameter $c$ and the success probability $p_c$. The $c$ and $p_c$ values are plotted along the $x$ and $y$ axes, respectively. The Euclidean distance of the limiting point (of Algorithm $1$) from the origin, i.e., $\lVert (x_{20000}(1), \ldots, x_{20000}(d)) \rVert$, is plotted along the $z$-axis. Each point on this plot is obtained by taking the average of $20$ independent experiments. It can be seen that Algorithm $1$ converges to a value close to the origin for {\bf smaller values of $c$} and {\bf larger values of $p_c$}. It converges farther away from the origin for {\bf larger values of $c$} and {\bf smaller values of $p_c$}. The sensitivity parameter $c$ seems to have a greater influence on the convergence-neighborhood than the success probability $p_c$. Finally, the surface of the plot is rough due to high variance of the gradient-estimation errors, see Section \ref{sec_bias}.
   
Figure \ref{fig3} illustrates the performance of DSPG in a $4$-agent setting with the sensitivity parameter $c$ fixed at $0.1$. The success probability $p_c$, plotted along $x$-axis, is varied between $0.3$ and $0.9$. The limit of DSPG, $x_{20000}$, is plotted along $y$-axis. We use different colors to represent different agents (blue for agent-$1$, orange for agent-$2$, green for agent-$3$ and red for agent-$4$). For example, at $p_c = 0.7$, agent-1 is at $-0.5 \exp(-18)$, agent-2 at $- 0.2 \exp(-18)$, agent-3 at $- 0.9 \exp(-18)$ and agent-4 is at $- 0.4 \exp(-18)$,
i.e., $x_{20000} = \left[-0.5 \exp(-18), - 0.2 \exp(-18), - 0.9 \exp(-18), - 0.4 \exp(-18) \right]$. As before, we plot the average over $20$ independent experiments. As compared to $p_c = 0.4$, the algorithm converges closer to the origin with $p_c = 0.9$.
   \begin{figure}[H]
      \centering
      \includegraphics[width=.5\textwidth]{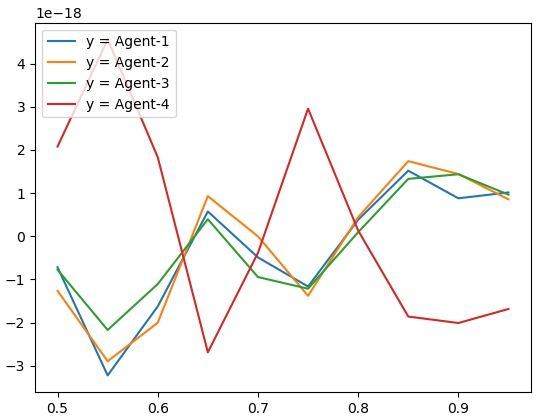}
      \caption{Average of $20$ experiments with $c = 0.1$. $p_c$ is plotted along $x$-axis and the limiting point is plotted along $y$-axis. Different colors are used for different agents.}
      \label{fig3}
   \end{figure}
   \begin{figure}[H]
      \centering
      \includegraphics[width=.5\textwidth]{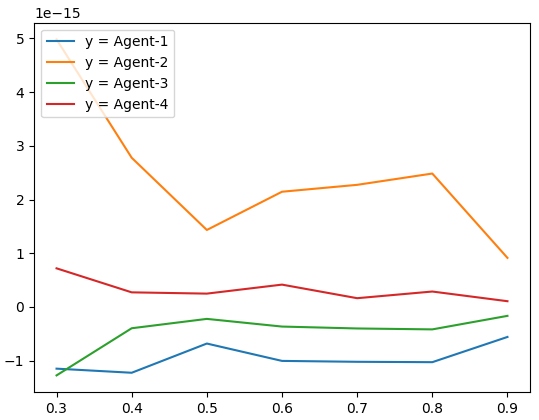}
      \caption{Same experimental setup as in Figure \ref{fig3} but with $c = 5$.}
      \label{fig4}
   \end{figure}
   For Figure \ref{fig4}, the experimental setup used in Figure \ref{fig3} is retained, with the exception that the sensitivity parameter $c$ is now set to $5$. As compared to $c=0.1$, DSPG converges farther from the origin when $c = 5$, for all values of $p_c$.
   \begin{figure}[H]
      \centering
      \includegraphics[width=.5\textwidth]{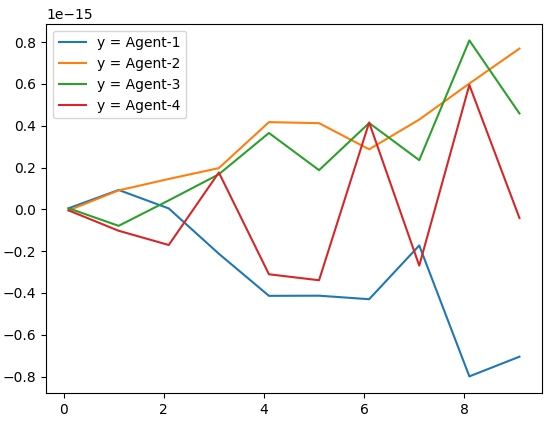}
      \caption{Average of $20$ experiments with $p_c = 0.3$. $c$ is plotted along $x$-axis and the limiting point is plotted along $y$-axis. Again, different colors are used for different agents.}
      \label{fig5}
   \end{figure}  
   \begin{figure}[H]
      \centering
      \includegraphics[width=.5\textwidth]{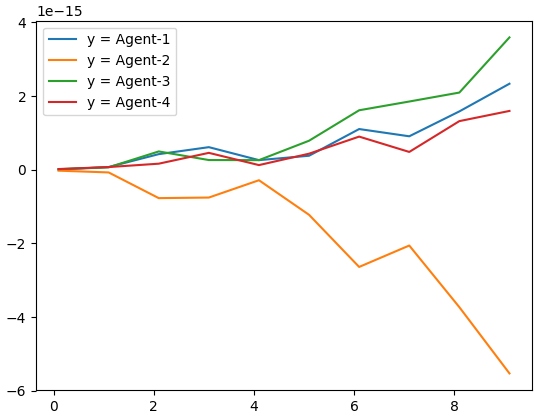}
      \caption{Same experimental setup as in Figure \ref{fig3} but with $p_c = 0.9$ and a {\bf constant step-size of 0.001}.}
      \label{fig6}
   \end{figure}
   Figures \ref{fig5} and \ref{fig6} illustrate the average performance of DSPG over $20$ experiments, for $p_c = 0.3$ and $0.9$, respectively. The sensitivity parameter $c$, plotted along $x$-axis, is varied between $0.1$ and $10$ in steps of $0.1$. As is expected, DSPG converges closer to the origin for smaller values of $c$. An important point to note is that the experiments illustrated in Figure \ref{fig6} were conducted using a constant step-size (learning rate) of $0.001$. While the analysis in Section \ref{sec_convergence} requires diminishing step-sizes, we observed that constant step-size algorithms converge, provided $c$ is sufficiently small. {\bf However, it must be noted that DSPG is not stable when using a constant learning rate}.
 \begin{remark}  
 Let us consider an implementation of DSPG using a constant learning rate.
 Recall from Section~\ref{sec_bias} that the gradient-estimation-bias is in $o(c^2 d)$. This appears as an additive error term in the descent direction. If we factor in the constant learning rate, say $\gamma$, then an error-term in $o(\gamma c^2 d)$ is added to the descent direction. This additive error can be controlled by choosing a smaller learning rate, however this slows the rate of convergence. Instead of using the learning rate to control the additive errors, it is efficient to control it through the use smaller values for the sensitivity parameter $c$. Since, this directly reduces the additive errors without affecting the rate of convergence.
 \end{remark}
   \subsection{Exp.2: $10$-agent system}
   Finally, we conducted similar experiments to understand the performance of DSPG in a $10$-agent setting. Results from these experiments are summarized in Figure \ref{fig2}. From the calculations in Section \ref{sec_bias} it is clear that higher variance is to be expected in a $10$-agent system as opposed to a $4$-agent one. This is certainly reflected in the plot wherein the worst performance is at $c = 8$ and $p_c = 0.7$. For very small values of $c$, the variance does remain low.
   \begin{figure}[H]
      \centering
      \includegraphics[width=.5\textwidth]{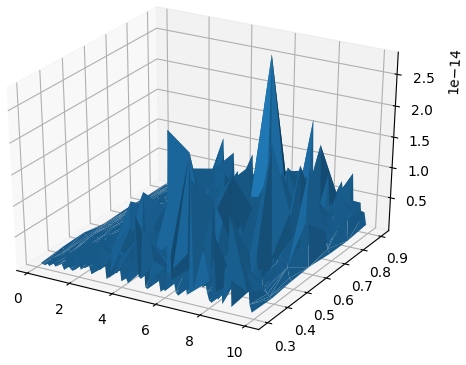}
      \caption{10-agent summary:  $0.1 \le c \le 10$ is plotted along $x$-axis, $0.3 \le p_c \le 0.9$ is plotted along $y$-axis and $\lVert (x_{20000}(1), \ldots, x_{20000}(10)) \rVert$ is plotted along  $z$-axis. Each point in the plot is obtained as an average of $10$ independent experiments.}
      \label{fig2}
   \end{figure}
   Let us quickly summarize the empirical results. DSPG converges to a neighborhood of the common minimizer. The size of this neighborhood is determined by the value of the sensitivity parameter $c$. With regards to variance, the sensitivity parameter has a greater influence than the success probability (of transmissions). Further, there is evidence that the variance is higher for larger size multi-agent systems. Finally, when implemented using constant learning rate DSPG still converges, however it is occasionally unstable.
   \section{DSPG AND THE CONSENSUS PROBLEM}
   Consensus is a common problem occurring in multi-agent systems. The cumulative consensus problem in a d-agent system involves solving the following in a cooperative manner:
   \begin{equation}
   \label{cons_eq}
   \underset{x \in \mathbb{R}^d}{\text{argmin }} \sum \limits_{i=1}^d f_i(x),
   \end{equation}
   where $f_i: \mathbb{R}^d \to \mathbb{R}$ is only accessible to agent-i, and $x = (x(1), \ldots, x(d))$ such that $x(i)$ is a local control variable of agent-i. The following is a standard assumption in cummulative consensus problems: \\ \vspace*{0.02cm}\\
   \textbf{(A)} $f_i$ is a convex function for $1 \le i \le d$. Further, there is a unique $x^*$ which solves the cumulative consensus problem. \\ \vspace*{0.02cm}\\
   Ideally, agent-i wishes to update its estimate of $x^*(i)$ as follows:
   \[
   x_{n+1}(i) = x_n(i) - \gamma_n \sum \limits_{i=1}^d \nabla_{x(i)} f_j (x_n), \text{ where}
   \]
   $x_n = (x_n(1), \ldots, x_n(d))$, and $\nabla_{x(i)} f_j (\cdotp)$ is the partial derivative of $f_j$ taken with respect to $x(i)$, i.e., the $i^{th}$ partial derivative of $f_j$. There is, however, a two fold problem with such an update: (i) agent-i does not have direct access to $f_j$ or $x(j)$ for $j \neq i$, and (ii) it does not account for the presence of an unreliable wireless communication network. To overcome this, we propose that at any given time $n$, agent-j collects the latest possible estimates of $x^*(k)$ for $k \neq j$, and calculates $\hat{g}^j _n (i)$ for $1 \le i \le d$, where
   \begin{equation} \label{cons_eq1}
   \hat{g}^j _n (i) = \frac{f_j \left(
   \begin{matrix}
   x_{n - \tau_{1j}(n)}(1) + c \Delta_{n}^j(1) \\
   \vdots \\
   x_{n - \tau_{dj}(n)}(d) + c \Delta_{n}^j(d)
   \end{matrix}
   \right) - f_j \left(
   \begin{matrix}
   x_{n - \tau_{1j}(n)}(1) - c \Delta_{n}^j(1) \\
   \vdots \\
   x_{n - \tau_{dj}(n)}(d) - c \Delta_{n}^j (d)
   \end{matrix} \right)}{2c \Delta_n^j(i)}.
   \end{equation}
   Here, $\tau_{ij}(n)$ is the delay random variable associated with obtaining the estimate of $x^*(i)$ from agent-i. Note that $\hat{g}^j _n (i)$ is an approximation for $\nabla_{x(i)}f_j(x_{n - \tau_{1j}(n)}(1), \ldots, x_{n - \tau_{dj}(n)}(d)$. The delays are due to the presence of an unreliable wireless network. For the sake of simplicity, the reader may assume that every pair of agents is connected by a bidirectional wireless communication link. Although in reality, this information may have to be relayed over a multi-hop network. Once calculated, agent-j sends $\hat{g}^j _n (i)$ to agent-i. For its part, agent-i collects $\{\hat{g}^j _n (i)\}_{\underset{j \neq i}{1 \le j \le d}}$ and performs the following update:
   \begin{equation} \label{cons_update}
   x_{n+1}(i) = x_n(i) - \gamma_n \left( \sum \limits_{j = 1}^d \hat{g}^j _{n - \hat{\tau}_{ji}(n)}(i) \right), \text{ where}
   \end{equation}
$\hat{g}^j _{n - \hat{\tau}_{ji}(n)}(i)$ is obtained from agent-j and $\hat{\tau}_{ji}(n)$ is the delay random variable associated with obtaining the gradient approximation along the i$^{th}$ direction from agent-j. 

It is clear that \eqref{cons_update} is roughly in the form of the DSPG update. Except that, in the current problem, there is a two stage information exchange between the agents. One can prove, using arguments similar to those in Section~\ref{sec_convergence}, that \eqref{cons_update} converges to $x^* = \underset{x \in \mathbb{R}^d}{\text{argmin}} \sum \limits_{i=1}^d f_i(x)$, under the assumptions listed in Section~\ref{sec_asmp}, provided we modify $(A5)$ by replacing the requirement that  $\mathbb{E} \tau_{ij}(n) ^2 < \infty$ for $1 \le i,j \le d$ and $n \ge 0$, with $\mathbb{E}[ \hat{\tau}_{ji}(n) + \tau_{kj}(n - \hat{\tau}_{ji}(n)) ]^2 < \infty$ for $1 \le i,j,k \le d$ and $n \ge 0$.

\section{CONCLUSIONS AND FUTURE WORKS}
In this paper, we presented DSPG, a decentralized approximate descent method for distributed optimization problems. 
The gradient estimator used in DSPG is cross-entropy based and operates merely through sampling. The bias due to sampling is shown to be in the order of the sensitivity parameter $c$. The variance is affected by both the sensitivity parameter and the number of agents in the system. However, this variance can be controlled by choosing a small value for $c$ and ensuring frequent communication between agents. We analyzed the convergence behavior of DSPG and showed that it converges to a small neighborhood of the common minimum. Further, we showed that this neighborhood depends on $c$. We also presented empirical evidence in support of the aforementioned theories.  We conducted experiements for a constant learning rate implementation of DSPG. We observed that although the algorithm converged, it was occasionally unstable. Finally, we briefly discussed how the consensus problem can be solved used the ideas presented herein.

In the future it would be interesting to explore the stability issues of constant step-size DSPG. It would also be interesting to calculate the rate of convergence of DSPG.


\bibliographystyle{plain}
\bibliography{reference}
\end{document}